\documentclass[a4paper]{article}    
\usepackage{amsthm,amssymb,amsmath,enumerate}
\usepackage{bbm}
\usepackage[all]{xy}
\usepackage{graphicx}

\advance\textheight by -1cm% to make page fit on my monitor at 2000
 
\newtheorem{definition}{Definition}
\newtheorem{proposition}[definition]{Proposition}
\newtheorem{theorem}[definition]{Theorem}

\newtheorem{lemma}[definition]{Lemma}

\newcommand{\fin}{\small\rm fin}

\newcommand{\C}{{\mathcal C}}

\newcommand{\E}{{\mathcal E}}

\newcommand{\RR}{{\mathbb R}}

\newcommand{\U}{{\mathcal U}}

\newcommand{\Z}{{\mathbb Z}}
\newcommand{\ZZ}{{\mathbb Z}}
\newcommand{\id}{\mathbbm{1}}
\newcommand{\interior}{{\rm int\>}}
\newcommand{\sm}{\setminus}
\newcommand{\ke}{{\rm Ker\>}}
\newcommand{\im}{{\rm Im\>}}

\let\sub=\subseteq
\let\subset=\subseteq

\let\phi=\varphi
\let\es=\emptyset
\newcommand\fes{(f_e)_\sharp}
\newcommand\restr{\!\restriction\!}
\newcommand{\fs}{f_\sharp}
\newcommand{\gs}{g_\sharp}

\newcommand{\noproof}{\unskip\nobreak\hfill\penalty50\hskip2em\hbox{}\nobreak\hfill%
       $\square$\parfillskip=0pt\finalhyphendemerits=0\par}
\newcommand{\COMMENT}[1]{}

\newcommand{\emtext}[1]{\text{\em #1}}

\newenvironment{txteq}
  {\begin{equation}\begin{minipage}[c]{0.8\textwidth}\em}
  {\end{minipage}\ignorespacesafterend\end{equation}\ignorespacesafterend}

\newcommand{\assign}{
  \mathrel{\mathop{:}}=
}
\newcommand{\sasign}{
  \mathrel{=\mathop{:}}
}

\def\lowfwd #1#2#3{{\setbox0\hbox{$#1$}\setbox1\hbox{$E'\!$}
            \mathchoice
            {{\mathop{\kern0pt #1}\limits^{\kern#2pt\raise.#3ex
     \vbox to 0pt{\hbox{$\scriptscriptstyle\rightarrow$}\vss}}}}%
            {{\mathop{\kern0pt #1}\limits^{\kern#2pt\raise.#3ex
     \vbox to 0pt{\hbox{$\scriptscriptstyle\rightarrow$}\vss}}}}%
            {\ifdim\wd0<\wd1{\,\vec{#1}\,}\else
     {\mathop{\kern0pt #1}\limits^{\kern#2pt\raise.0ex
     \vbox to 0pt{\hbox{$\scriptscriptstyle\rightarrow$}\vss}}}\fi}%
            {{\vec{#1}}}%
            }}
\def\fwd #1#2{{\lowfwd{#1}{#2}{15}}}
\def\lowbkwd #1#2#3{{\mathop{\kern0pt #1}\limits^{\kern#2pt\raise.#3ex
     \vbox to 0pt{\hbox{$\scriptscriptstyle\leftarrow$}\vss}}}}

\def\vC{\kern-1pt\fwd C3\kern-.5pt}

\def\vCC{\kern-.7pt\fwd{\C}3\kern-.7pt}
\def\vd{\kern-1pt\lowfwd d2{10}\kern-1pt}
\def\vD{\kern-.7pt\fwd D3\kern-.5pt}
\def\ve{\kern-1pt\lowfwd e{1.5}1\kern-1pt}
\def\vf{\kern-1pt\lowfwd f{1.5}1\kern-1pt}
\def\fv{\kern-1pt\lowbkwd f{1.5}1\kern-1pt}
\def\ev{\kern-1pt\lowbkwd e{1.5}1\kern-1pt}
\def\veStar{{\mathop{\kern0pt e\lower1.5pt\hbox{${}^*$}}\limits^{\kern0pt
   \raise.02ex\vbox to 0pt{\hbox{$\scriptscriptstyle\rightarrow$}\vss}}}}
\def\eStarv{{\mathop{\kern0pt e\lower1.5pt\hbox{${}^*$}}\limits^{\kern0pt
   \raise.02ex\vbox to 0pt{\hbox{$\scriptscriptstyle\leftarrow$}\vss}}}}
\def\vedash{{\mathop{\kern0pt e\lower.5pt\hbox{${}% logically \v(e')
     \scriptstyle'$}}\limits^{\kern0pt\raise.02ex
     \vbox to 0pt{\hbox{$\scriptscriptstyle\rightarrow$}\vss}}}}

\def\vE{\kern-.7pt\fwd E3\kern-.7pt}

\def\vEE{\kern-.7pt\fwd{\E}3\kern-.7pt}
\def\vF{\kern-.7pt\fwd F3\kern-.7pt}

\def\vG{\kern-.7pt\fwd G3\kern-.7pt}
\def\vH{\kern-.5pt\fwd H3\kern-.5pt}
\def\vP{\kern-.7pt\fwd P3\kern-.6pt}

\def\specrel#1#2{\mathrel{\mathop{\kern0pt #1}\limits_{#2}}}

\newcommand{\ee}{\text{e}}
\newcommand{\ii}{\text{i}}

\title{\hbox{On the homology of locally compact spaces with ends}}

\author{Reinhard Diestel and Philipp Spr\"ussel}
\date{}

\begin{document}      

\maketitle

\begin{abstract}
We propose a homology theory for locally compact spaces with ends in~which the ends play a special role. The approach is motivated by results for graphs with ends, where it has been highly successful. But it was unclear how the original graph-theoretical definition could be captured in the usual language for homology theories, so as to make it applicable to more general spaces. In this paper we provide such a general topological framework: we define a homology theory which satisfies the usual axioms, but which maintains the special role for ends that has made this homology work so well for graphs.
\end{abstract}

\smallskip

\section{Introduction}

The first homology group of a finite graph~$G$, known in graph theory as its \emph{cycle space}, is an important aspect in the study of graphs and their properties. Although the groups that occur tell us little as such---they are always a sum of $\Z$s depending only on the number of vertices and edges of~$G$---they way the interact with the combinatorial structure of~$G$ has implications for commonly investigated graph properties such as planarity or duality.

For the simplicial homology of infinite graphs these standard theorems fail, but this can be remedied: they do work with the cycle space $\C(G)$ constructed in~\cite{DiestelBook05, CyclesI}, as amply demonstrated e.g. in~\cite{BergerBruhnDeg, LocFinTutte, Duality, bicycle, Partition, LocFinMacLane, DiestelBook05, CyclesIntro, hotchpotch, geo, Arboricity}. This space is built not from finite (elementary) cycles in $G$ itself, as in simplicial homology, but from the (possibly infinite) edge sets of topological circles in the Freudenthal compactification $|G|$ of~$G$, obtained from $G$ by adding its ends. The definition of $\C(G)$ also allows for locally finite infinite sums.

Given the success of $\C(G)$ for graphs, it seems desirable to recast its definition in homological terms that make no reference to the one-dimensional character of~$|G|$ (e.g., to circles), to obtain a homology theory for similar but more general spaces (such as non-compact CW complexes of any dimension) that implements the ideas and advantages of $\C(G)$ more generally. Simplicial homology is easily seen not to be the right approach. One way of extending simplicial homology to more general spaces is \v{C}ech homology; and indeed its first group applied to~$|G|$ turns out to be isomorphic to~$\C(G)$. But there the usefulness of \v{C}ech homology ends: since its groups are constructed as limits rather than directly from chains and cycles, they do not interact with the combinatorial structure of $G$ in the way we expect and know it from~$\C(G)$~\cite{Hom1}. We therefore adopt a singular approach.

On the face of it, it is not clear whether $\C(G)$ might in fact be isomorphic, even canonically, to the first singular homology group $H_1(|G|)$ of $|G|$. However, it was shown in~\cite{Hom1} that it is not: surprisingly, $\C(G)$ is always a natural quotient of~$H_1(|G|)$, but this quotient is proper unless $G$ is essentially finite. Thus, $\C(G)$~is a genuinely new object, also from a topological point of view.

In this paper, we shall define a homology theory that satisfies all the usual axioms and will work for any locally compact Hausdorff space~$X$ given with a fixed (Hausdorff) compactification~$\hat X$. For compact $X=\hat X$ our homology will coincide with the standard singular homology.%
   \COMMENT{}
  For non-compact~$X$, it will be `larger' than the homology of $X$ itself, but `smaller' than the singular homology of~$\hat X$.
  When $X$ is a graph and $\hat X = |X|$ is its Freudenthal compactification, its first group will be canonically isomorphic to the cycle space $\C(X)$ of~$X$.

The main idea of our homology, and how it comes to sit `between' the homologies of $X$ and of~$\hat X$, is that we use the compactification points, or \emph{ends}, differently from other points. Ends will be allowed as inner points of simplices, but not as vertices of simplices. The chains we use, which may be infinite, have to be locally finite in~$X$ but not around ends. Note that we use the term `end' loosely here, for any point in $\hat X\sm X$. These may be ends in the usual sense; but we allow other situations too, such as boundary points of hyperbolic groups etc.%
   \COMMENT{}

As a one-dimensional example of what to expect as the (intended) outcome, consider the infinite 1-chain $\sum_{i\in\Z} \sigma_i$ in the space~$\RR$, where $\sigma_i\colon [0,1]\to [i,i+1]$ maps $x$ to~$i+x$. This chain has zero boundary, but there are many good reasons why we do not want to allow it as a 1-cycle.%
   \COMMENT{}
   Now add edges $e_i$ from $i$ to~$-i$, for every integer~$i$. In the new space obtained, the same chain $\sum_{i\in\Z} \sigma_i$ will now be a welcome 1-cycle. The reason is that its simplices now form a circle: the addition of the new edges has resulted in the two ends of $\RR$ being identified into one end. Hence in the new ambient space our chain can be viewed as a single loop subdivided infinitely often. We shall want infinite subdivision to be possible within a homology class, and thus our chain must now be equivalent to that loop. The challenge in setting up our homology theory will lie in how to allow ends to influence and shape the homology indirectly, as in this example, while at the same time meeting the formal axioms for a homology theory that make no reference to an ambient space.

  This paper is organized as follows. After giving the basic definitions in Section~\ref{sec:term}, we discuss a preliminary version of our new homology in Section~\ref{sec:old}. This version is very simple to define, and for graphs it already captures the cycle space. However, it falls short of one of the usual axioms for homology theories, the `long exact sequence' axiom. This is remedied in Section~\ref{sec:new}, where we refine the definition of our new homology. We then show that it satisfies the axioms for homology and that for spaces $|X|$ with $X$ a graph it coincides with the cycle space $\C(X)$.

\section{Terminology and basic facts}\label{sec:term}

We use the terminology of~\cite{DiestelBook05} for graphs and that of~\cite{Hatcher} for topology. Our graphs may have multiple edges but no loops. This said, we shall from now on use the term \emph{loop} topologically: for a topological path $\sigma\colon [0,1]\to X$ with $\sigma(0) = \sigma(1)$. This loop is \emph{based at} the point~$\sigma(0)$.

Let us define the (topological) cycle space $\C$ of a locally finite graph~$G$. This is usually defined over~$\Z_2$ (which suffices for its role in graph theory), but we wish to prove our main results more generally with integer coefficients. (The $\Z_2$ case will easily follow.) We therefore need to speak about orientations of edges.

An edge $e=uv$ of a locally finite graph $G$ has two \emph{directions}, $(u,v)$ and~$(v,u)$.%
   \COMMENT{}
A~triple $(e,u,v)$ consisting of an edge together with one of its two directions is an \emph{oriented edge}. The two oriented edges corresponding to $e$ are its two \emph{orientations}, denoted by $\ve$ and~$\ev$. Thus, $\{\ve,\ev\} = \{(e,u,v), (e,v,u)\}$, but we cannot generally say which is which.%
   \COMMENT{}
However, from the definition of $G$ as a CW-complex we have a fixed homeomorphism $\theta_e\colon [0,1]\to e$. We call $(\theta_e(0),\theta_e(1))$ the \emph{natural direction} of~$e$, and $(e,\theta_e(0),\theta_e(1))$ its \emph{natural orientation}.

Let $\vEE = \vEE(G)$ denote the set of all integer-valued functions $\phi$ on the set $\vE$ of all oriented edges of $G$ that satisfy $\phi(\ev) = -\phi(\ve)$ for all $\ve\in \vE$. This is an abelian group under pointwise addition. A~family $(\phi_i\mid i\in I)$ of elements of $\vEE$ is \emph{thin} if for every $\ve\in\vE$ we have $\phi_i(\ve)\ne 0$ for only finitely many~$i$. Then $\phi = \sum_{i\in I} \phi_i$ is a well-defined element of~$\vEE$: it maps each $\ve\in\vE$ to the (finite) sum of those $\phi_i(\ve)$ that are non-zero. We shall call a function $\phi\in\vEE$ obtained in this way a \emph{thin sum} of those~$\phi_i$.

The (topological) \emph{cycle space} $\C(G)$ of $G$ is the subgroup of $\vEE$ consisting of all thin sums of maps $\vEE\to\Z$ defined naturally by the oriented \emph{circles} in the Freudenthal compactification $|G|$ of $G$, the homeomorphic images in $|G|$ of the (oriented) circle $S^1$.

As already mentioned, this notion of the cycle space enables us to generalize all the usual finite cycle space theorems to locally finite graphs. One basic fact that we will need later is the following. A set $\vF\subset\vEE$ is an \emph{oriented cut} if there is a vertex set $X$ such that $\vF$ is the set of all oriented edges from $X$ to $V\sm X$, i.e.\ oriented edges $(e,x,y)$ with $x\in X$ and $y\in V\sm X$.

\begin{lemma}\label{orthogonal}
  An element of the edge space of a locally finite graph $G$ lies in the cycle space if and only if its values on the edges of every finite oriented cut of $G$ sum to zero.
\end{lemma}
See~\cite{DiestelBook05} for a proof of the (unoriented) $\Z_2$-version of Lemma~\ref{orthogonal}. It adapts readily to the version stated here.

For an oriented edge $\ve$ of $G$ and a path $\sigma:[0,1]\to|G|$, a \emph{pass} of $\sigma$ through $\ve$ is a restriction of $\sigma$ to a subinterval $[a,b]$ of $[0,1]$ such that $\sigma(a)$ is the first and $\sigma(b)$ the last vertex of $\ve$ and such that $(a,b)$ is mapped to the interior of $e$.

The \emph{standard $n$-simplex}
\begin{equation*}
  \Big\{(t_0,\dotsc,t_n)\in\RR^{n+1} \bigm| \sum_i t_i=1 \text{ and } t_i\ge 0\text{ for all } i\Big\}
\end{equation*}
is denoted by $\Delta^n$. Given points $v_0,\dotsc,v_n\in\RR^m$ (not necessarily in general position),%
   \COMMENT{}
we write $[v_0,\dotsc,v_n]$ for their convex hull.%
   \COMMENT{}
The \emph{natural map} $\Delta^n\to[v_0,\dotsc,v_n]$ is the linear map $(t_0,\dotsc,t_n)\mapsto\sum t_iv_i$.

If $v_0,\dotsc,v_n$ are in general position, then the natural map $\Delta^n\to[v_0,\dotsc,v_n]$ is clearly a homeomorphism. Then $[v_0,\dotsc,v_n]$ is an \emph{$n$-simplex in $\RR^m$}, the point $v_i$ is its \emph{$i$th vertex}. Every convex hull of $k+1\le n$ vertices is a \emph{$k$-face} of $[v_0,\dotsc,v_n]$. We use $[v_0,\dotsc,\hat v_i,\dotsc,v_n]$ to denote the ($n-1$)-face spanned by all the vertices but $v_i$.

A \emph{singular $n$-simplex} in a topological space $X$ is a continuous map $\Delta^n\to X$. A \emph{$k$-face} of a singular $n$-simplex $\sigma$ is a map $\tau=\sigma\restr D$, where $D$ is a $k$-face of $\Delta^n$.

Given a set $\{X_k \mid i\in I\}$ of topological spaces, we write $X=\bigsqcup X_k$ for their disjoint union endowed with the disjoint union topology.%
   \COMMENT{}

A \emph{homology theory} assigns to every space $X$ and every subspace $A$ of $X$ a sequence $\big(H_n(X,A)\big)_{n\in\Z}$ of abelian groups\footnote{Usually $H_n(X,A)$ is the trivial group for $n<0$.}, and to every continuous map $f:X\to Y$ with $f(A)\subset B$ for subspaces $A$ of $X$ and $B$ of $Y$ (which we indicate by writing $f:(X,A)\to(Y,B)$) a sequence $f_*:H_n(X,A)\to H_n(Y,B)$ of homomorphisms, such that $(fg)_*=f_*g_*$ for compositions of maps and $\id_*=\id$ for the identity maps. We abbreviate $H_n(X,\es)$ to $H_n(X)$. Finally, the following \emph{axioms for homology} have to be satisfied:
\begin{description}
\item[Homotopy equivalence:]
  If continuous maps $f,g:(X,A)\to(Y,B)$ are homotopic, then $f_*=g_*$.
\item[The Long Exact Sequence of a Pair:]
  For each pair $(X,A)$ there are \emph{boundary homomorphisms} $\partial:H_n(X,A)\to H_{n-1}(A)$ such that
  \begin{equation*}
    \begin{xy}
     \xymatrix{
	\dotsb \ar[r]^{\partial\quad} & H_n(A) \ar[r]^{\iota_*} & H_n(X) \ar[r]^{\pi_*\;\;} & H_n(X,A) \ar[dl]_{\partial} & \\
	& & H_{n-1}(A) \ar[r]^{\iota_*} & H_{n-1}(X) \ar[r]^{\quad\pi_*} & \dotsb
      }
    \end{xy}
  \end{equation*}
  is an exact sequence, where $\iota$ denotes the inclusion $(A,\es)\to (X,\es)$ and $\pi$ denotes the inclusion $(X,\es)\to (X,A)$. These boundary homomorphisms are \emph{natural}, i.e.\ given a continuous map $f:(X,A)\to(Y,B)$ the diagrams
  \begin{equation*}
    \begin{xy}
      \xymatrix{
        H_n(X,A) \ar[r]^{\partial} \ar[d]^{f_*} & H_{n-1}(A) \ar[d]^{f_*}\\
	H_n(Y,B) \ar[r]^{\partial} & H_{n-1}(B)
      }
    \end{xy}
  \end{equation*}
  commute.
\item[Excision:]
  Given subspaces $A,B$ of $X$ whose interiors cover $X$, the inclusion $(B,A\cap B)\hookrightarrow(X,A)$ induces isomorphisms $H_n(B,A\cap B)\to H_n(X,A)$ for all $n$.
\item[Disjoint unions:]
  For a disjoint union $X=\bigsqcup_{\alpha}X_\alpha$ with inclusions $\iota_{\alpha}:X_{\alpha}\hookrightarrow X$, the direct sum map $\bigoplus_{\alpha}\left(\iota_{\alpha}\right)_*:\bigoplus_{\alpha}H_n(X_{\alpha},A_{\alpha}) \to H_n(X,A)$, where $A=\bigsqcup_{\alpha}A_{\alpha}$, is an isomorphism.
\end{description}

The original Eilenberg-Steenrod axioms~\cite{EilSteen} contain an additional axiom, called the `dimension axiom', stating that the homology groups of a single point are nonzero only in dimension zero. However, this axiom is not always regarded as an essential part of the requirements for a homology theory~\cite{Hatcher}. An example for a homology theory that does not satisfy the dimension axiom is bordism theory; in this case the groups of a single point are nontrivial in infinitely many dimensions. We omit the dimension axiom, but note that the homology theory we construct will trivially satisfy it.

The groups $H_n(X,A)$ above are called \emph{relative homology groups}; specializations $H_n(X)=H_n(X,\es)$ are \emph{absolute homology groups}.

\section{An ad-hoc modification of singular homology for locally compact spaces with ends}\label{sec:old}

In this section we describe an ad-hoc way to define homology groups that extend the main properties of the cycle space of graphs to arbitrary dimensions. The purpose of this section is to introduce the main ideas needed for the homology we shall define in Section~\ref{sec:new} in a technically simpler setting.

Let $X$ be a locally compact Hausdorff space and let $\hat X$ be a Hausdorff compactification of $X$. (See e.g.~\cite{AbelsStrantzalos} for more on such spaces.)%
   \COMMENT{}
Note that every locally compact Hausdorff space is Tychonoff, and thus has a Hausdorff compactification. The kind of spaces we have in mind is that $X$ is a locally finite CW-complex and $\hat X$ is its Freudenthal compactification, but formally we do not make any further assumptions. Nevertheless, we will call the points in $\hat X\sm X$ \emph{ends}, even if they are not ends in the usual, more restrictive, sense.

  Although our chains, cycles etc.\ will live in~$\hat X$, we shall denote their groups as $C_n(X)$, $Z_n(X)$ etc, with reference to $X$ rather than~$\hat X$: this is because ends will play a special role, so the information of which points of $\hat X$ are ends must be encoded in the notation for those groups.

Let us call a family $(\sigma_i\mid i\in I)$ of singular $n$-simplices in $\hat X$ \emph{admissible} if
\begin{enumerate}[(i)]
\item $(\sigma_i\mid i\in I)$ is locally finite in~$X$, that is, every $x\in X$ has a neighbourhood in $X$ that meets the image of $\sigma_i$ for only finitely many~$i$;%
   \COMMENT{}
   \item every $\sigma_i$ maps the $0$-faces of~$\Delta^n$ to~$X$.%
   \COMMENT{}
   \end{enumerate}
Note that as $X$ is locally compact, (i)~is equivalent to asking that every compact subspace of $X$%
   \COMMENT{}
meets the image of $\sigma_i$ for only finitely many~$i$.%
   \COMMENT{}
Condition~(ii), like~(i), underscores that ends are not treated on a par with the points in~$X$: we allow them to occur on infinitely many~$\sigma_i$ (which (i) forbids for points of~$X$), but not in the fundamental role of images of 0-faces: all simplices must be `rooted' in~$X$. If $X$ is%
   \COMMENT{}
a countable union of compact spaces, (i)~and (ii) together imply that admissible families are countable, i.e.\ that $|I|\le\aleph_0$.%
   \COMMENT{}

When $(\sigma_i\mid i\in I)$ is an admissible family of $n$-simplices, any formal linear combination $\sum_{i\in I} \lambda_i \sigma_i$ with all $\lambda_i\in\Z$ is an \emph{$n$-sum in~$X$}.\footnote{In standard singular homology, one does not usually distinguish between formal sums and chains. It will become apparent soon why we have to make this distinction.}%
   \COMMENT{}
We regard $n$-sums $\sum_{i\in I}\lambda_i\sigma_i$ and $\sum_{j\in J}\mu_j\tau_j$ as \emph{equivalent} if for every $n$-simplex $\rho$ we have $\sum_{i\in I, \sigma_i=\rho}\lambda_i = \sum_{j\in J, \tau_j=\rho}\mu_j$. Note that these sums are well-defined since an $n$-simplex can occur only finitely many times in an admissible family. We write $C_n(X)$ for the group of \emph{$n$-chains}, the equivalence classes of $n$-sums. The elements of an $n$-chain are its \emph{representations}. Clearly every $n$-chain $c$ has a unique (up to re-indexing) representation whose simplices are pairwise distinct---which we call the \emph{reduced representation} of $c$---, but we shall consider other representations too.%
   \COMMENT{}
The subgroup of $C_n(X)$ consisting of those $n$-chains that have a finite representation is denoted by $C'_n(X)$.

The boundary operators $\partial_n\colon C_n\to C_{n-1}$ are defined by extending linearly from~$\partial_n\sigma_i$, which are defined as usual in singular homology.%
   \COMMENT{}
Note that $\partial_n$ is well defined (i.e., that it preserves the required local finiteness),%
    \COMMENT{}
and $\partial_{n-1}\partial_n = 0$.%
   \COMMENT{}
Chains in $\im\partial$ will be called \emph{boundaries}.

As $n$-cycles, we do \emph{not} take the entire kernel of~$\partial_n$. Rather, we define $Z'_n(X) := \ke (\partial_n\restr C'_n(X))$,%
   \COMMENT{}
and let $Z_n(X)$ be the set of those $n$-chains that are sums of such finite cycles:%
   \COMMENT{}
\begin{equation*}
  Z_n (X) \assign \Big\{\phi\in C_n(X)\Bigm| \phi = \sum_{j\in J} z_j
    \emtext{ with } z_j\in Z'_n(X)\ \forall j\in J\Big\}\,.
\end{equation*}
More precisely,%
   \COMMENT{}
an $n$-chain $\phi\in C_n(X)$ shall lie in $Z_n(X)$ if it has%
   \COMMENT{}
a representation $\sum_{i\in I}\lambda_i\sigma_i$ for which $I$ admits a partition%
   \COMMENT{}
into finite sets~$I_j$ ($j\in J$) such that, for every $j\in J$, the $n$-chain $z_j \in C'_n(X)$ represented by $\sum_{i\in I_j} \lambda_i\sigma_i$ lies in $Z'_n(X)$. Any such representation of $\phi$ as a formal sum will be called a \emph{standard representation} of~$\phi$ \emph{as a cycle}.%
\footnote{Since the $\sigma_i$ need not be distinct, $\phi$~has many representations by formal sums. Not all of these need admit a partition as indicated---an example will be given later in the section.}%
   \COMMENT{}
We call the elements of $Z_n(X)$ the \emph{$n$-cycles} of~$X$.

The chains in $B_n(X) := \im\partial_{n+1}$ then form a subgroup of~$Z_n(X)$: by definition, they can be written as $\sum_{j\in J}\lambda_jz_j$ where each $z_j$ is the (finite) boundary of a singular ($n+1$)-simplex.%
   \COMMENT{}
We therefore have homology groups
\begin{equation*}
  H_n(X) := Z_n(X)/B_n(X)
\end{equation*}
as usual.

Note that if $X$ is compact, then all admissible families and hence all chains are finite,%
   \COMMENT{}
so the homology defined above coincides with the usual singular homology.
The characteristic feature of this homology is that while infinite cycles are allowed, they are always of `finite character': in any standard representation of an infinite cycle, every finite subchain is contained in a larger finite subchain that is already a cycle.%
   \COMMENT{}

For graphs and Freudenthal compactifications, the finite character of this homology is also shown in another aspect: It is shown in~\cite{Hom1} that every $1$-cycle---finite or infinite---is homologous to a cycle whose reduced representation consists of a single loop.% We thus have:

%\begin{proposition}\label{oldH1finite}
%  If $X$ is a graph and $\hat X$ its Freudenthal compactification, then every class in $H_1(X)$ is represented by a finite cycle.
%\end{proposition}

Let us now define relative homology groups $H_n(X,A)$. Normally, these groups are defined for all subsets $A\subset X$. In our case, the subspace $A$ has to satisfy further conditions. Since we wish to consider chains in $A$, in our sense, $A$ has to be locally compact and come with a compactification $\hat A$. Chains in $A$ have to be chains also in $X$, hence we further need that $\hat A\subset \hat X$, and that ends of $A$ lie in $\hat X\sm X$, that is, they have to be ends of $X$.%
   \COMMENT{}

Let $A$ be a \emph{closed} subset of $X$ (but not necessarily closed in $\hat X$). Since $X$ is locally compact, so is $A$.%
   \COMMENT{}
Let $\hat A$ denote the closure of $A$ in $\hat X$. Then $\hat A$ is a compactification of $A$, and $\hat A\sm A \subset \hat X\sm X$.%
   \COMMENT{}
Clearly, admissible families of simplices in $A$ are also admissible in $X$.%
   \COMMENT{}
We define $H_n(X,A)$ as follows. Let $C_n(X,A)$ be the quotient group $C_n(X)/C_n(A)$,%
   \footnote{Formally, $C_n(A)$ is not a subset of $C_n(X)$, because the equivalence classes of $n$-sums in $X$ are larger than those in~$A$. For instance, every formal sum $\sigma-\sigma$ with $\sigma$ a singular $n$-simplex in $X$ that does not live in $A$ is part of the equivalence class of the empty $n$-sum in $X$, but not in $A$. 
   But there is a natural embedding $C_n(A)\hookrightarrow C_n(X)$: map an $n$-chain in $A$ to the $n$-chain in $X$ with the same reduced representation.}
and let $C'_n(X,A)$ be the subgroup of all its elements $\phi+C_n(A)$ with $\phi\in C'_n(X)$. Define $Z'_n(X,A)$ as the kernel of the quotient map $C_n(X,A)\to C_{n-1}(X,A)$ of $\partial_n$ restricted to $C'_n(X,A)$, and $B_n(X,A)$ as the image of the quotient map $C_{n+1}(X,A)\to C_n(X,A)$ of $\partial_{n+1}$. Then define $Z_n(X,A)$ from $Z'_n(X,A)$ as before, and put $H_n(X,A)=Z_n(X,A)/B_n(X,A)$. Clearly, $H_n(X,\es)=H_n(X)$.

Let us look at an example. For simplicity, we will restrict our attention to absolute homology. Consider the \emph{double ladder}. This is the $2$-ended graph $G$ with vertices $v_n$ and~$v'_n$ for all integers~$n$, and with edges $e_n$ from $v_n$ to~$v_{n+1}$, edges $e'_n$ from $v'_n$ to~$v'_{n+1}$, and edges $f_n$ from $v_n$ to~$v'_n$. The 1-simplices corresponding to these edges, oriented in their natural directions, are $\theta_{e_n}$, $\theta_{e'_n}$ and~$\theta_{f_n}$, see Figure~\ref{fig:double}.

\begin{figure}[htbp]
  \centering
  \includegraphics[width=0.7\linewidth]{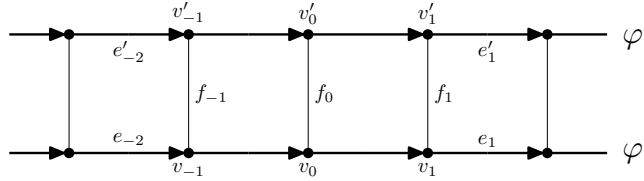}
  \caption{The $1$-chains $\phi$ and $\phi'$ in the double ladder.}
  \label{fig:double}
\end{figure}

In order to let the elements of our homology be defined, let $\hat G$ be any Hausdorff compactification of $G$. (One could, for instance, choose the Freudenthal compactification $|G|$ of $G$.) For the infinite chains $\phi$ and $\phi'$ represented by $\sum \theta_{e_n}$ and $\sum \theta_{e'_n}$, respectively, and for $\psi := \phi-\phi'$ we have $\partial\phi = \partial\phi' = \partial\psi = 0$, and neither sum as written above contains a finite cycle. However, we can rewrite $\psi$ as $\psi = \sum z_n$ with finite cycles $z_n = \theta_{e_n} + \theta_{f_{n+1}} - \theta_{e'_n} - \theta_{f_n}$. This shows that $\psi\in Z_1(G)$, although this was not visible from its original representation.

By contrast, one can show that $\phi\notin Z_1(G)$. This follows from Theorem~\ref{oldH1isC} below and the known characterizations of $\C(G)$~\cite[Theorem~8.5.8]{DiestelBook05}, but is not obvious. For example, one might try to represent $\phi$ as $\phi = \sum_{n=1}^{\infty} z'_n$ with $z'_n \assign \theta_{e_{-n}} + \theta_{n-1} + \theta_{e_n} - \theta_n$, where $\theta_n:[0,1]\to e_{-n}\cup\dotsb\cup e_n$ maps $0$ to $v_{-n}$ and $1$ to $v_{n+1}$, see Figure~\ref{fig:single}.

\begin{figure}[htbp]
  \centering
  \includegraphics[width=0.7\linewidth]{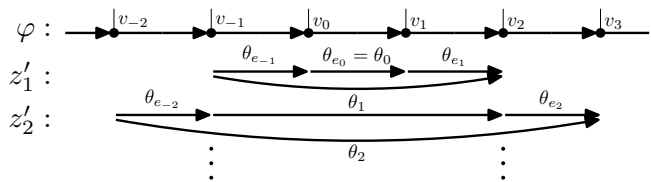}
  \caption{Finite cycles summing to~$\phi$---by an inadmissible sum.}
  \label{fig:single}
\end{figure}

This representation of $\phi$, however, although well defined as a formal sum (since every simplex occurs at most twice), is not a legal $1$-sum, because its family of simplices is not locally finite and hence not admissible. (The point $v_0$, for instance, lies in every simplex $\theta_i$.)

\begin{theorem}[\cite{Hom1}]\label{oldH1isC}
  If $X$ is a locally finite connected graph and $\hat X$ its
  Freudenthal compactification, then there is a natural group
  isomorphism from $H_1(X)$ to the cycle space $\C(X)$. Moreover,
  every class of $H_1(X)$ has a finite representative.
\end{theorem}

%The isomorphism $f$ in Theorem~\ref{oldH1isC} counts, for each oriented edge $e$ of $G$, how often the simplices in a representative of a given homology class traverse it, and lets the image of this class assign this sum to $e$. See~\cite{Hom1} for a formal definition
%of $f$.

While Theorem~\ref{oldH1isC} shows that the homology defined in this section succeeds in capturing $\C(G)$, it is not a homology theory: It fails to allow for long exact sequences as demanded by the axioms. To see this, let $A\subset X$ consist of a single point $a$ in $X$ and assume there is a path $\pi$ in $\hat X$ from $a$ to an end. The $0$-chain $c=-\sigma$ in $A$, where $\sigma:\{0\}\to A$, is a $0$-cycle whose homology class in $H_0(A)$ lies in the kernel of $\iota_*:H_0(A)\to H_0(X)$ (because $c=\partial\tau$ for $\tau = \sum_{i=1}^{\infty}\pi\restr[1-2^{1-i},1-2^{-i}]$) but not in the image of $\partial_1:H_1(X,A)\to H_0(A)$ (because clearly no finite $1$-cycle in $X$ can have boundary $c$, and no infinite $1$-cycle in $X$ that is a sum of finite cycles can have boundary $c$, since by Condition~(i) only finitely many of those finite cycles meet $a$). Hence the long sequence for the pair $(X,A)$ fails to be exact at $H_0(A)$.

\section{A new homology for locally compact spaces with ends}\label{sec:new}

In this section we define a homology theory that implements the same ideas as our ad-hoc homology of Section~\ref{sec:old}, but which will satisfy all the usual axioms. To achieve this, we shall encode all the properties we shall need into the definition of chains---rather than restricting both chains and cycles, as in Section~\ref{sec:old}. Our homology will also be defined for disjoint unions of compactifications,%
   \COMMENT{}
i.e.\ for $X=\bigsqcup X_k$ and $\hat X=\bigsqcup \hat X_k$ where each $\hat X_k$ is a compactification of $X_k$. Nevertheless, we will start with the definition for compact $\hat X$ and then extend it to unions of compactifications.%
   \COMMENT{}

Let $X$ be a locally compact Hausdorff space, and let $\hat X$ be a Hausdorff compactification of $X$. We define admissible families and $n$-sums as in Section~\ref{sec:old}. All other notation will now be defined differently.

In order to capture $\C(G)$ in dimension $1$ for locally finite graphs, we have to consider chains consisting of infinitely many simplices. On the other hand, if one allows infinite chains without further restrictions, one obtains cycles like $\phi$ in Figure~\ref{fig:double}, which does not correspond to an element of $\C(G)$. The solution to this dilemma is to allow infinitely many simplices only if they are of a certain type.

Call a singular $n$-simplex $\sigma$ in $\hat X$ \emph{degenerate} if it is lower dimensional in the following sense: There is a compact Hausdorff space $X_{\sigma}$ of dimension at most $n-1$ such that $\sigma$ can be written as the composition of continuous maps $\Delta^n \to X_{\sigma} \to \hat X$.%
   \COMMENT{}
Recall that a normal space\footnote{Note that $X_{\sigma}$ is normal as it is compact and Hausdorff.} has dimension $k$ if and only if every (finite) open covering $\U$ has a refinement $\U'$ for which every point lies in at most $k+1$ sets of $\U'$, and $k$ is the least such number.%
   \COMMENT{}

As the empty space is the only space of dimension $-1$, and every $0$-dimensional space is totally disconnected, we have that no singular $0$-simplex is degenerate and a singular $1$-simplex is degenerate if and only if it is constant.

Denote by $C'_n(X)$ the group of equivalence classes of $n$-sums.%
   \COMMENT{}
(Recall that two $n$-sums are called equivalent if every $n$-simplex appears equally often---taking account of the multiplicities $\lambda_i$---in both sums.) As before, the elements of a class $c\in C'_n(X)$ are its \emph{representations}, its unique (up to re-indexing) representation $\sum\lambda_i\sigma_i$ with pairwise distinct $\sigma_i$ is the \emph{reduced representation} of $c$. Sums in $C'_n(X)$ are (well) defined in the obvious way as the equivalence class of the sum of any choice of representations of each of the summands. We call $c$ \emph{good} if the simplices $\sigma_i$ in its reduced representation are degenerate for all but finitely many $i\in I$. An \emph{$n$-chain in $X$} is an equivalence class $c\in C'_n(X)$ that can be written as $c = c_1+\partial c_2$,%
   \COMMENT{}
where both $c_1\in C'_n(X)$ and $c_2\in C'_{n+1}(X)$ are good. In other words, $c$ is an $n$-chain if and only if it has a representation $\sum_{i\in I}\lambda_i\sigma_i$ for which $I$ is the disjoint union of a finite set $I_0$, a (possibly infinite) set $I_1$, and finite sets $I_j,~j\in J,$ such that each $\sigma_i,~i\in I_1,$ is degenerate, and each sum $\sum_{i\in I_j}\lambda_i\sigma_i$ is the boundary of a degenerate singular ($n+1$)-simplex. We call such a representation a \emph{standard representation} of $c$. Note that a standard representation will not, in general, be a reduced representation, and vice versa, a reduced representation does not have to be standard.%
   \COMMENT{}

We write $C_n(X)$ for the group of all $n$-chains in~$X$. As usual, we write $Z_n(X):=\ke\partial_n$ and $B_n(X):=\im\partial_{n+1}$. The elements of $Z_n$ are \emph{$n$-cycles}, those of $B_n$ are \emph{boundaries}. Clearly, $B_n\sub Z_n$, so we can define the \emph{homology groups} $H_n(X):=Z_n/B_n$ as usual.

Since a cycle $c_1+\partial c_2$ as above represents the same homology class as $c_1$ does, we have at once:
\begin{proposition}\label{representHn}
  Every homology class is represented by a good $n$-cycle.%
   \COMMENT{}
\end{proposition}

As no singular $0$-simplex is degenerate, this means that every homology class in $H_0(X)$ is represented by a finite $0$-cycle. Moreover, as every degenerate $1$-simplex is constant and hence equivalent, as a $1$-sum, to the boundary of a constant (and thus degenerate) $2$-simplex, we have the same in dimension $1$:%
   \COMMENT{}
\begin{proposition}\label{representH1}
  Every homology class in $H_0(X)$ or in $H_1(X)$ is represented by a finite cycle.%
   \COMMENT{}
\end{proposition}

Let us now define relative homology groups. Consider a closed subset $A$ of $X$ and write $\hat A$ for the closure of $A$ in $\hat X$. In order to make all the axioms work, we additionally require the boundary of $\hat A$ in $\hat X$ to be a (compact) subset of $X$. In the case of graphs and their Freudenthal compactification, this is the case for instance if $A$ is a component of the graph minus a finite vertex set. In infinite graph theory, it is an often used procedure to contract such components, so it does not seem too restrictive to only consider such subsets. We call $(X,A)$ an \emph{admissible pair}. Like in Section~\ref{sec:old}, we obtain that admissible families of simplices in $A$ are admissible also in $X$. Now let $C_n(X,A)=C_n(X)/C_n(A)$, let $Z_n(X,A)$ be the kernel of the quotient map $C_n(X,A)\to C_{n-1}(X,A)$ of $\partial_n$, and $B_n(X,A)$ the image of the quotient map $C_{n+1}(X,A)\to C_n(X,A)$ of $\partial_{n+1}$, and define $H_n(X,A)=Z_n(X,A)/B_n(X,A)$.%
   \COMMENT{}

Having defined the homology groups for compactifications, we now extend it to disjoint unions of compactifications as follows: If $X=\bigsqcup X_k$, $\hat X=\bigsqcup \hat X_k$, and $A$ is a closed subspace of $X$ such that for each $k$ the pair $(X_k,A_k)$ is admissible, where $A_k\assign A\cap X_k$, we call $(X,A)$ an \emph{admissible pair}. For an admissible pair $(X,A)$, define $C_n(X,A)$ as the direct sum $\bigoplus C_n(X_k,A_k)$.%
   \COMMENT{}
(Note that each $A_k$ is closed in $X_k$.) The homology groups $H_n(X)$ and $H_n(X,A)$ are then defined in the obvious way.

Our earlier definitions of admissible families, $n$-sums, and $n$-chains for compact $\hat X$ also extend naturally to disjoint unions $X=\bigsqcup X_k$ as follows: A family of singular $n$-simplices in $X$ is \emph{admissible} if its subfamily of simplices in $X_k$ is admissible for finitely many $k$ and empty for all other $k$.%
   \COMMENT{}
An \emph{$n$-sum in $X$} is a formal sum $\sum\lambda_i\sigma_i$ where $(\sigma_i)$ is an admissible family. The equivalence classes of $n$-sums form a group $C'_n(X)$, an element $c$ of $C'_n(X)$ is \emph{good} if it has a representation in which all but finitely many simplices are degenerate, and an \emph{$n$-chain in $X$} is a class $c\in C'_n(X)$ that can be written as $c=c_1+\partial c_2$ with good $c_1\in C'_n(X)$ and good $c_2\in C'_{n+1}(X)$. It is easy to see that $C_n(X)$, defined earlier as $\bigoplus_k C_n(X_k)$, is indeed the group of $n$-chains in $X$.%
   \COMMENT{}
   
In standard homology, it is trivial that a chain in $X$ all of whose simplices live in $\hat A$ is also a chain in $A$. In our case, this is not immediate: If all simplices in the reduced representation of a chain in $X$ live in $\hat A$, this does not imply directly that there is a \emph{standard} representation that consists of simplices in $\hat A$. Indeed, if there is an infinite admissible family of degenerate ($n+1$)-simplices that do not live in $\hat A$ but whose boundaries do, then the sum of their boundaries is the representation of an $n$-chain in $X$, and all simplices in the reduced representation of this chain live in $\hat A$. But as soon as this reduced representation consists of infinitely many non-degenerate $n$-simplices, we do not know whether it does also represent a chain in $A$. Here it comes in that $(X,A)$ is an admissible pair: As each $\hat A_k$ has a compact boundary that is contained in $X_k$, there is no admissible family as above. An one can indeed show that a chain in $X$ is a chain in $A$ as soon as their simplices live in $\hat A$. More generally, we have the following.

\begin{lemma}\label{chaininA}
  Let $\sum_{i\in I}\lambda_i\sigma_i$ be a reduced representation of a chain in $X$ and let $I'\sub I$ be the set of those indices with $\im\sigma_i\sub\hat A$. Then $\sum_{i\in I'}\lambda_i\sigma_i$ is the reduced representation of a chain in $A$.
\end{lemma}

\begin{proof}
  Let $c$ be the $n$-chain in $X$ represented by $\sum_{i\in I}\lambda_i\sigma_i$. Choose a standard representation of $c$, i.e.
  \begin{equation}\label{representc}
    c = \sum_{i\in I_0}\lambda_i\sigma_i + \sum_{i\in I_1}\lambda_i\sigma_i + \sum_{j\in J}\lambda_j\partial\tau_j,
  \end{equation}
  where $I_0$ is finite and each simplex $\sigma_i,~i\in I_1,$ and $\tau_j,~j\in J,$ is degenerate. Note that not all simplices occurring in this representation have to live in $\hat A$, this only has to hold for the $n$-simplices that are part of the reduced representation of $c$.

  Let $I'_0\sub I_0$, $I'_1\sub I_1$, and $J'\sub J$ be the index sets of those simplices that live in $\hat A$. Let further $(\sigma_k)_{k\in K}$ be the family of those $n$-simplices living in $\hat A$ that are a face of some $\tau_j$ with $j\in J\sm J'$, and let $\lambda_k$ be the multiplicity in which $\sigma_k$ occurs in the sum $\sum_{j\in J\sm J'}\lambda_j\partial\tau_j$. Note that $K$ is finite since every $\tau_j,~j\in J\sm J'$, with a face $\sigma_k,~k\in K$, meets the compact boundary of $\hat A$ and $(\tau_j)_{j\in J\sm J'}$ is admissible. Now the $n$-sum
  \begin{equation*}
    \sum_{i\in I'_0}\lambda_i\sigma_i + \sum_{i\in I'_1}\lambda_i\sigma_i + \sum_{j\in J'}\lambda_j\partial\tau_j + \sum_{k\in K}\lambda_k\sigma_k
  \end{equation*}
  is a standard representation of a chain $c'$ in $A$, and by construction the reduced representation of $c'$ is precisely $\sum_{i\in I'}\lambda_i\sigma_i$.%
   \COMMENT{}
\end{proof}

Before we show in the following section that this is indeed a homology theory let us first note that, applied to a locally finite graph in dimension $1$, it captures precisely its cycle space:

\begin{theorem}\label{newH1isC}
  If $X$ is a locally finite connected graph and $\hat X$ is its Freudenthal compactification, then there is a canonical isomorphism $H_1(X)\to\C(X)$.
\end{theorem}

We will prove Theorem~\ref{newH1isC}, in fact a stronger statement,
in Section~\ref{sec:graphs} using results from
Section~\ref{sec:axioms}.

\section{Verification of the axioms}\label{sec:axioms}

In this last section we show that the homology theory we defined satisfies the axioms cited at the end of Section~\ref{sec:term}. As a preliminary step we have to show that continuous functions between spaces induce homomorphisms between their homology groups. This will not work for arbitrary continuous functions: As we distinguish between ends and other points, our functions will have to respect this distinction.

Let locally compact Hausdorff spaces $X,\hat X$ and $Y,\hat Y$ be given, where $X=\bigsqcup X_k$ and $\hat X=\bigsqcup\hat X_k$ with $\hat X_k$ a compactification of $X_k$, and similarly for $Y=\bigsqcup Y_l$. Let $A\sub X$ and $B\sub Y$ be closed subspaces such that $(X,A)$ and $(Y,B)$ are admissible pairs. As before, we write $\hat A$ and $\hat B$ for the closures of $A$ in $\hat X$ and $B$ in $\hat Y$, and note that $\hat A\sm A\sub \hat X\sm X$ and $\hat B\sm B\sub \hat Y\sm Y$. Let us call a continuous function $f:\hat X\to\hat Y$ a \emph{standard map} if $f(X)\subset Y$ and $f(\hat X\sm X)\subset \hat Y\sm Y$.%
   \COMMENT{}
If, in addition, $f(A)\subset B$, we write $f:(X,A)\to(Y,B)$. (As before, we refer to $X$ even though the functions live on $\hat X$.)

Let us show that every standard map $f:(X,A)\to(Y,B)$ induces a homomorphism $f_*:H_n(X,A)\to H_n(Y,B)$, defined as follows. For a homology class $[c]\in H_n(X,A)$, choose a standard representation $\sum_{i\in I}\lambda_i\sigma_i$ of $c$ and map $[c]$ to the homology class in $H_n(Y,B)$ that contains the $n$-cycle represented by $\sum_{i\in I}\lambda_i f\sigma_i$. In ordinary singular homology this map is always well defined. To see that it is well defined in our case, note first that $f$ preserves the equivalence of sums, maps boundaries to boundaries, and maps degenerate simplices to degenerate simplices. Hence all that remains to check is that $f$ maps chains to chains. The following lemma implies that it does:

\begin{lemma}\label{standardgood}
  For every standard map $f:(X,A)\to(Y,B)$, if $(\sigma_i)_{i\in I}$ is an admissible family of $n$-simplices in $\hat X$ (resp.\ $\hat A$), then $(f\sigma_i)_{i\in I}$ is an admissible family of $n$-simplices in $\hat Y$ (resp.\ $\hat B$).
\end{lemma}

\begin{proof}
  As $f$ is standard and $(\sigma_i)_{i\in I}$ is admissible, every $f\sigma_i$ maps the $0$-faces of $\Delta^n$ to $Y$. It therefore remains to show that every $y\in Y$ has a neighbourhood that meets the image of $f\sigma_i$ for only finitely many $i$. Let $U$ be a compact neighbourhood of $y$ in $Y$, its preimage $f^{-1}(U)$ is a subset of $X=\bigsqcup X_k$ that is closed in $\hat X$ as $f$ is continuous.%
   \COMMENT{}
Hence $f^{-1}(U)\cap\hat X_k$ is compact for each $k$.%
   \COMMENT{}
Since $(\sigma_i)_{i\in I}$ is admissible, it contains simplices in only finitely many~$\hat X_k$, and as the subfamilies of simplices in those $\hat X_k$ are admissible, only finitely many $\sigma_i$ meet $f^{-1}(U)$.%
   \COMMENT{}
Hence $U$ meets the image of $f\sigma_i$ for only finitely many $i$ and hence $(f\sigma_i)_{i\in I}$ is admissible. The analogous claim for $A$ and $B$ follows as $f(A)\subset B$.
\end{proof}

By Lemma~\ref{standardgood} the map $\sum_{i\in I}\lambda_i\sigma_i \mapsto \sum_{i\in I}\lambda_i f\sigma_i$ defines a homomorphism \begin{equation*}
  \fs:C_n(X,A)\to C_n(Y,B)
\end{equation*}
with $\partial\fs=\fs\partial$. Thus every standard map $f:(X,A)\to(Y,B)$ induces a homomorphism $f_*:H_n(X,A)\to H_n(Y,B)$. It is easy to see that if $g:(Y,B)\to(Z,C)$ is another standard map we have $(fg)_*=f_*g_*$, and that $\id_*=\id$.

We thus have shown that our homology admits induced homomorphisms if the continuous functions satisfy the natural condition that they map ends to ends and points in $X$ to points in $Y$. We now show that, subject only to similarly natural constraints,%
   \COMMENT{}
our homology satisfies the axioms for a homology theory.

We will verify the axioms in the order of Section~\ref{sec:term}.

\begin{theorem}[Homotopy equivalence]\label{generalHomEq}
  If standard maps $f,g:(X,A)\to(Y,B)$ are homotopic via standard maps $(X,A)\to(Y,B)$ then $f_*=g_*$.
\end{theorem}

\begin{proof}
  Denote by $F=(f_t)_{t\in[0,1]}$ the homotopy between $f$ and $g$ consisting of standard maps $f_t:(X,A)\to(Y,B)$ and satisfying $f_0=f$ and $f_1=g$. We first consider the absolute groups $H_n(X)$, $H_n(Y)$.

  The main ingredient in the proof of homotopy equivalence for standard singular homology is a decomposition of $\Delta^n\times[0,1]$ into ($n+1$)-simplices $D_0,\dotsc,D_n$ (see eg.~\cite{Hatcher}). This decomposition works as follows: Let $\Delta^n\times\{0\}=[v_0,\dotsc,v_n]$ and $\Delta^n\times\{1\}=[w_0,\dotsc,w_n]$, and put $D_j:=[v_0,\dotsc,v_j,w_j,\dotsc,w_n]$. Each $D_j$ is an ($n+1$)-simplex, and hence the natural map between $\Delta^{n+1}$ and $D_j$ is a homeomorphism which we denote by $\tau_j$.

  In standard singular homology, for an $n$-chain $z=\sum_{i\in I}\lambda_i\sigma_i$ in $X$ one considers the ($n+1$)-chain
  \begin{equation} \label{prism}
    P(z)=\sum_{i\in I}\sum_{j=0}^{n}(-1)^j\lambda_i F\circ(\sigma_i\times\id)\circ\tau_j
  \end{equation}
  in $Y$, where $\sigma\times\id:\Delta^n\times[0,1]\to X\times[0,1]$ is the map $(a,b)\mapsto(\sigma(a),b)$, and then shows that $\partial P(z)+P(\partial z)=\gs(z)-\fs(z)$. If $z$ is an $n$-cycle, then $\gs(z)-\fs(z)=\partial P(z)+P(\partial z)=\partial P(z)$, thus $\gs(z)-\fs(z)$ is a boundary, which means that $\gs$ and $\fs$ take $z$ to the same homology class and hence $f_*([z])=g_*([z])$.
  
  In our case, we first have to show that, given an $n$-chain $z$ in $X$ with representation $\sum_{i\in I}\lambda_i\sigma_i$, the expression $P(z)$ in~\eqref{prism} is indeed an ($n+1$)-sum, i.e.\ that $(F\circ(\sigma_i\times\id)\circ\tau_j)_{i\in I, j\in\{0,\dotsc,n\}}$ is an admissible family of ($n+1$)-simplices in $\hat Y$. Then we have to show that the $c\in C'_{n+1}$ represented by $P(z)$ has a standard representation. If these two claims are true, we will also have $\partial P(z)+P(\partial z)=\gs(z)-\fs(z)$ and hence $f_*([z])=g_*([z])$.
  
  To show that the family $(F\circ(\sigma_i\times\id)\circ\tau_j)_{i\in I, j\in\{0,\dotsc,n\}}$ is admissible, note first that, since $(\sigma_i)_{i\in I}$ is an admissible family of simplices in $\hat X$, their images meet only finitely many $\hat X_k$; let $\hat X^-$ be their (compact) union. Now let $y\in Y$ be given, and choose a compact neighbourhood $U$ of $y$. As $\hat Y$ is Hausdorff, $U$ is closed in $\hat Y$. Consider the preimage of $U$ under $F$. As $U$ is closed and $F$ is continuous, this is a closed subset of $\hat X^-\times[0,1]$, and hence compact. Its projection
  \begin{equation*}
    \tilde U:=\{x\in\hat X \mid \exists t\in[0,1]: F(x,t)\in U\}
  \end{equation*}
  to $\hat X^-$, then, is also compact. Since $U\subset Y$ and each $f_t$ is standard, we have $\tilde U\subset X$, so $\tilde U$ meets $\im\sigma_i$ for only finitely many $i$. And for only those $i$ does $U$ meet the image of any $F\circ(\sigma_i\times\id)\circ\tau_j,~j\in\{0,\dotsc,n\}$. Hence $P(z)$ is an ($n+1$)-sum.%
   \COMMENT{}

  To verify our second claim, let $[z]\in H_n(X)$ be given, and assume without loss of generality that $z$ is good (cf.\ Proposition~\ref{representHn}), i.e.\ it has a representation $\sum_{i\in I}\lambda_i\sigma_i$ such that only finitely many of the $\sigma_i$ are not degenerate. We show that if $\sigma_i$ is degenerate then $F\circ(\sigma_i\times\id)\circ\tau_j$ is degenerate for each $j$; from this it follows directly that $P(z)$ as stated in~\eqref{prism} is a standard representation of an ($n+1$)-chain in $Y$.
  
  Suppose that $\sigma_i$ is degenerate; then there is a compact Hausdorff space $X_{\sigma_i}$ of dimension at most $n-1$, and continuous maps $\alpha:\Delta^n\to X_{\sigma_i}$ and $\beta:X_{\sigma_i}\to \hat X$ with $\sigma_i = \beta\circ\alpha$. Now let $\gamma:\Delta^n\to X_{\sigma_i}\times[0,1]$ be the composition of the natural map from $\Delta^n$ to $D_j\subset \Delta^n\times[0,1]$ and the map $\alpha\times\id$ from $\Delta^n\times[0,1]$ to $X_{\sigma_i}\times[0,1]$. Then $F\circ(\sigma_i\times\id)\circ\tau_j = \beta\circ\gamma$, so all that remains to show is that $X_{\sigma_i}\times[0,1]$ has dimension at most $n$. But this is immediate by the fact that $X_{\sigma_i}$ has dimension at most $n-1$, $[0,1]$ has dimension $1$, and that the dimension of the product of two compact spaces is at most the sum of their dimensions~\cite[Theorem~3.2.13]{Engelking}.

  We thus have $f_*=g_*:H_n(X)\to H_n(Y)$. As $P$ takes sums in $A$ to sums in $B$, the formula $\partial P+P\partial=\gs-\fs$ remains valid also for relative chains, and thus we also have $f_*=g_*:H_n(X,A)\to H_n(Y,B)$.
\end{proof}

\begin{theorem}[The Long Exact Sequence of a Pair]\label{exact}
  There are boundary homomorphisms $\partial:H_n(X,A)\to H_{n-1}(A)$ such that
  \begin{equation*}
    \begin{xy}
      \xymatrix{
	\dotsb \ar[r]^{\partial\quad} & H_n(A) \ar[r]^{\iota_*} & H_n(X) \ar[r]^{\pi_*\;\;} & H_n(X,A) \ar[r]^{\partial} & H_{n-1}(A) \ar[r]^{\quad\iota_*} & \dotsb
      }
    \end{xy}
  \end{equation*}
  is an exact sequence, where $\iota$ denotes the inclusion $(A,\es)\to (X,\es)$ and $\pi$ denotes the inclusion $(X,\es)\to (X,A)$. These boundary homomorphisms are \emph{natural}, i.e.\ given a standard map $f:(X,A)\to(Y,B)$ the diagrams
  \begin{equation*}
    \begin{xy}
      \xymatrix{
        H_n(A,X) \ar[r]^{\partial} \ar[d]^{f_*} & H_{n-1}(A) \ar[d]^{f_*}\\
	H_n(Y,B) \ar[r]^{\partial} & H_{n-1}(B)
      }
    \end{xy}
  \end{equation*}
  commute.
\end{theorem}

\begin{proof}
  As clearly $\im\iota_{\sharp}=\ke\pi_{\sharp}$ we have a short exact sequence of chain complexes
  \begin{equation*}
    \begin{xy}
      \xymatrix{
        & 0 \ar[d] & 0 \ar[d] & 0 \ar[d] & \\
        \cdots \ar[r]^{\partial\quad\;\;} & C_{n+1}(A) \ar[r]^{\;\partial} \ar[d]^{\iota_{\sharp}} & C_n(A) \ar[r]^{\partial\;} \ar[d]^{\iota_{\sharp}} & C_{n-1}(A) \ar[r]^{\quad\;\;\partial} \ar[d]^{\iota_{\sharp}} & \cdots \\
        \cdots \ar[r]^{\partial\quad\;\;} & C_{n+1}(X) \ar[r]^{\;\partial} \ar[d]^{\pi_{\sharp}} & C_n(X) \ar[r]^{\partial\;} \ar[d]^{\pi_{\sharp}} & C_{n-1}(X) \ar[r]^{\quad\;\;\partial} \ar[d]^{\pi_{\sharp}} & \cdots \\
        \cdots \ar[r]^{\partial\quad\;\;} & C_{n+1}(X,A) \ar[r]^{\;\partial} \ar[d] & C_n(X,A) \ar[r]^{\partial\;} \ar[d] & C_{n-1}(X,A) \ar[r]^{\quad\;\;\partial} \ar[d] & \cdots \\
        & 0 & 0 & 0 &
      }
    \end{xy}
  \end{equation*}
  It is a general algebraic fact (see eg.~\cite{Hatcher}) that for every short exact sequence of chain complexes there exists a natural boundary homomorphism $\partial$ of the corresponding homology groups giving the desired long exact sequence.
\end{proof}

\begin{theorem}[Excision]\label{excision}
  Let $(X,A)$ be an admissible pair and let $B$ be a closed subset of $X$ such that the interiors $\interior\hat A$ of $\hat A$ and $\interior\hat B$ of $\hat B$ cover $\hat X$. Then the inclusion $(B,A\cap B)\hookrightarrow(X,A)$ induces isomorphisms $H_n(B,A\cap B)\to H_n(X,A)$.
\end{theorem}

To prove Theorem~\ref{excision}, we first sketch the proof of excision for ordinary singular homology, and then point out the differences to our case. We start with barycentric subdivision of simplices.%
   \COMMENT{}
The aim is to find a sufficiently fine barycentric subdivision so as to construct a homomorphism from $C_n(X)$ to $C_n(A+B):= C_n(A)+C_n(B)\sub C_n(X)$.

\begin{lemma}\label{barycentric}
  For every $n$-simplex $[v_0,\dotsc,v_n]$ there is a finite family of degenerate simplices%
   \COMMENT{}
  $\Delta^{n+1}\to[v_0,\dotsc,v_n]$ such that adding the boundaries of those ($n+1$)-simplices, as well as the $n$-simplices in the corresponding families of the ($n-1$)-faces of $[v_0,\dotsc,v_n]$, to the natural map $[v_0,\dotsc,v_n]$ yields the a sum of simplices in its barycentric subdivision (with suitable signs).
\end{lemma}

\begin{proof}
  Induction on $n$. The lemma is clearly true for $n=0$. For $n>0$, let $b$ be the barycentre of $[v_0,\dotsc,v_n]$. Then $[v_0,\dotsc,v_n]$ is homologous to $\sum_{k=0}^{n}(-1)^k\Delta_k$ with $\Delta_k\assign[b,v_0,\dotsc,\hat v_k,\dotsc,v_n]$, since it differs from this sum by the boundary of the degenerate ($n+1$)-simplex $[b,v_0,\dotsc,v_n]$. By induction, every ($n-1$)-face of $[v_0,\dotsc,v_n]$ is homologous via boundaries of degenerate simplices to a sum of the simplices in its barycentric subdivision plus a sum of (degenerate) simplices for each of its ($n-2$)-faces. Hence $\Delta_k$, being the cone over the ($n-1$)-face $[v_0,\dotsc,\hat v_k,\dotsc,v_n]$ is a corresponding sum of boundaries of degenerate simplices one dimension higher.%
   \COMMENT{}
  As each ($n-2$)-face appears equally often as a face of an ($n-1$)-face of $[v_0,\dotsc,v_n]$ with positive and negative sign, so does the sum of (degenerate) simplices belonging to this face. Hence those sums cancel in the sum of all boundaries, which implies that $[v_0,\dotsc,v_n]$ is homologous to a sum of the desired type.
\end{proof}

\noindent
For every singular $n$-simplex $\sigma$, let $T(\sigma)$ be the sum consisting of the compositions of $\sigma$ and each of the degenerate ($n+1$)-simplices provided by Lemma~\ref{barycentric} applied to $\Delta^n$ and let $S(\sigma)$ be the sum of restrictions of $\sigma$ to the simplices in the barycentric subdivision of $\Delta^n$. Then Lemma~\ref{barycentric} says that (with appropriate choice of the signs in $T$ and $S$)
\begin{equation*}
  \partial T(\sigma) = \sigma - T(\partial\sigma) - S(\sigma).
\end{equation*}

Now $S$ and $T$ extend to a chain map $S:C_n(X)\to C_n(X)$, that is, a map with $\partial S=S\partial$ (which follows immediately from the definition of the barycentric subdivision), and a map $T:C_n(X)\to C_{n+1}(X)$ with
\begin{equation}\label{chainST}
  \partial T + T\partial = \id - S.
\end{equation}

Next, let us define, for every positive integer $m$, the map $D_m:C_n(X)\to C_{n+1}(X)$ like in standard homology, i.e.\ $D_m:= \sum_{0\le j<m}TS^j$. Note that
\begin{equation}\label{chainDmSm}
  \partial D_m + D_m\partial = \id - S^m
\end{equation}
by~\eqref{chainST} and the fact that $S$ is a chain map.%
   \COMMENT{}

Finally, define maps $D:C_n(X)\to C_{n+1}(X)$ and $\rho:C_n(X)\to C_n(A+B)$ as follows: For every singular simplex $\sigma$, let $m(\sigma)$ be the smallest number $m$ for which every simplex in $S^m(\sigma)$ lives in the interior of $\hat A$ or of $\hat B$. Now define $D(\sigma):=D_{m(\sigma)}$ and extend linearly to $C_n(X)$. The map $\rho$ is defined by $\rho(\sigma):= S^{m(\sigma)}(\sigma) + D_{m(\sigma)}(\partial\sigma) - D(\partial\sigma)$ and extending linearly. Note that $\rho(\sigma)$ is indeed in $C_n(A+B)$, see~\cite{Hatcher}. With this notation, we have
\begin{equation}\label{chainDrho}
  \partial D + D\partial = \id - \iota\rho,
\end{equation}
where $\iota$ is the inclusion $C_n(A+B)\to C_n(X)$. Moreover, we clearly have
\begin{equation}\label{chainrhoiota}
  \rho\iota = \id.
\end{equation}

In the case of our homology, we have to confront three major problems in order to define $D$ and $\rho$ so as to satisfy~\eqref{chainDrho} and~\eqref{chainrhoiota}:\footnote{In order to avoid confusion with the notation of the case of standard homology, we will from now label the maps from standard homology by adding the index $_{\fin}$.} Firstly, these maps will map a singular simplex to a sum of simplices, but the underlying family of this sum need not be admissible as its simplices may map $0$-faces to ends. Hence we have to change the maps so that the simplices in their image map $0$-faces to $X$. The second problem is that, while we change the image simplices, we have to ensure that each of them still lives in the interior of $\hat A$ or of $\hat B$. Hence we are not allowed to change them too much. The third problem will be to guarantee that the image of a chain is a chain, i.e.\ that it has a standard representation. We shall overcome the first two problems by subdividing the simplices at points that are mapped to $X$ contrary to the barycentres of $\Delta^n$ and its faces.

To make this precise, we define the notion of a \emph{$\sigma$-pseudo-linear} $m$-simplex, where $\sigma$ is a given singular $n$-simplex. Let points $w_0,\dotsc,w_m,w'_0,\dotsc,w'_m\in\Delta^n,~m\ge 1$, be given such that $\sigma$ maps each $w'_i$ to $X$ and each $w_i$ with $w_i\not= w'_i$ to $\hat X\sm X$. The $\sigma$-pseudo-linear $m$-simplex with \emph{centre} $[w_0,\dotsc,w_m]$ and \emph{antennae} $w_iw'_i,~0\le i\le m,$ is a singular simplex $\tau:\Delta^m \to [w_0,\dotsc,w_m] \cup \bigcup_{i=0}^m w_iw'_i$ defined as follows. Let $v_0,\dotsc,v_m$ be the vertices of $\Delta^m$ and consider the following simplex $[v'_0,\dotsc,v'_m]\sub\Delta^m$: Put $v'_i\assign v_i$ if $w_i=w'_i$ and $v'_i\assign\frac{1}{m+2}\left(2v_i+\sum_{j\not= i}v_j\right)$ otherwise. Then map $[v'_0,\dotsc,v'_m]$ to $[w_0,\dotsc,w_m]$ by sending $v'_i$ to $w_i$ and extending linearly, and map each line $v'_iv_i$ to the line $w_iw'_i$. Call the union of $[v'_0,\dotsc,v'_m]$ and the $v'_iv_i$ the \emph{kernel} of $\Delta^m$ with respect to the points $w_i$ and $w'_i$.

For $m=1$, this already defines the simplex $\tau$. For $m>1$ and each $l$-face of $[w_0,\dotsc,w_m]$ ($1\le l<m$) define $\tau$ on the kernel of this face (with respect to the $w_i$ in this face and the associated $w'_i$) the same way it is defined on the kernel of $\Delta^m$. Now consider a point $x$ on the boundary of $[v'_0,\dotsc,v'_m]$. For every face of $[v'_0,\dotsc,v'_m]$ that contains $x$, we say that the projection of $x$ to the corresponding face of $\Delta^m$ is \emph{associated with $x$}. Note that this point lies in the kernel of this face of $\Delta^m$ and that $\tau$ maps it to the same point in $\Delta^n$ as $x$. Together with $x$ these points span internally disjoint simplices as follows: For every maximal descending sequence of faces that contain $x$,%
   \COMMENT{}
the points on those faces associated with $x$ span a simplex whose dimension only depends on the dimension of the smallest face that contains $x$.%
   \COMMENT{}
For a point on some line $v'_iv_i$ we obtain a set of ($n-1$)-simplices defined in the same way.%
   \COMMENT{}
It is easy to see that these simplices are disjoint for distinct points $x,x'$ and that they cover all of $\Delta^m$ apart from the interior of its kernel. We can thus define $\tau$ on each such simplex as the contstant function with image the image of $x$.

The definition of $\sigma$-pseudo-linear simplices immediately yields that the boundary of a $\sigma$-pseudo-linear ($m+1$)-simplex $\tau$ is the sum (with appropriate signs) of the $\sigma$-pseudo-linear $m$-simplices with centres the $m$-faces of the centre of $\tau$ (and the corresponding antennae). This implies

\begin{lemma}\label{antennaesub}
  If an $m$-simplex $[w_0,\dotsc,w_m]\sub\Delta^n$ is homologous to a sum of $m$-simplices, then this remains true if we choose a point $w'$ for every vertex $w$ of those simplices and replace each simplex $S$ by a $\sigma$-pseudo-linear simplex with centre $S$ and antennae all lines from a vertex $w$ of $S$ to its $w'$.\noproof
\end{lemma}

The maps $D$ and $\rho$ will map a singular simplex $\sigma$ to a sum consisting of compositions of $\sigma$ and $\sigma$-pseudo-linear simplices, and correspondingly a chain $c$ to a sum of compositions with $\sigma$-pseudo-linear simplices for all simplices $\sigma$ in a representation of $c$ still to be chosen. In order to chose the antennae of the $\sigma$-pseudo-linear simplices, we shall use a subset $B'$ of $B$ defined as follows: For every point in the boundary of $\hat A$, choose a compact neighbourhood that is contained in $B$.%
   \COMMENT{}
Since the boundary of each $\hat A_k=\hat A \cap \hat X_k$ is compact, finitely many such neighbourhoods suffice to cover it. Let $B'$ be the union of $\hat B \sm \hat A$ and the neighbourhoods for all $k$. Write $\hat B'$ for the closure of $B'$ in $\hat X$. Note that the interiors of $\hat A$ and $\hat B'$ cover $\hat X$ and that the boundaries of each $\hat B'_k = \hat B' \cap \hat X_k$ is a compact subset of $X_k$.

Now consider a singular $n$-simplex $\sigma$. Let $b$ be the barycentre of $\Delta^n$. If $\sigma(b)\in X$, then we set $b':=b$. Otherwise consider the line $bv_0$, where $\Delta^n=[v_0,\dotsc,v_n]$. As $\hat X\sm X$ is closed and $\sigma$ is continuous, there is a last point $\tilde b$ on this line  for which $\sigma(b\tilde b)\sub \hat X\sm X$. Since the boundaries of $\hat A$ and $\hat B$ are contained in $X$, we can find a point $b'$ on the line $bv_0$ so that
\begin{txteq}\label{interiorA}
  if $\sigma(b)$ lies in the interior of $\hat A$ then so does $\sigma(bb')$
\end{txteq}
and
\begin{txteq}\label{interiorB}
  if $\sigma(b)$ lies in the interior of $\hat B'$ then so does $\sigma(bb')$.
\end{txteq}
Proceed analogously if $b$ is a barycentre of a face of $\Delta^n$. The only difference is that we consider the line $bv_j$, where $j$ is the smallest index with $v_j$ belonging to that face. It is not hard to see that the points $b'$ can be chosen so that, for singular simplices with a common face, the choices of the points on this face coincide.%
   \COMMENT{}

We are now ready to define the maps $D$ and $\rho$. For every singular simplex $\sigma$, let $m(\sigma)$ be the smallest number $m$ for which every simplex in $S^m(\sigma)$ lives in the interior of $\hat A$ or of $\hat B'$. Now for a chain $c\in C_n(X)$ with reduced representation $c = \sum_{i\in I}\lambda_i\sigma_i$, consider the sum
\begin{equation*}
  \sum_{i\in I}\lambda_i (D_{m(\sigma_i)})_{\fin}(\sigma_i)
\end{equation*}
and define $D(c)$ to be the sum obtained from the above sum by replacing each simplex in each $(D_{m(\sigma_i)})_{\fin}(\sigma_i)$ by the composition of $\sigma_i$ and a $\sigma_i$-pseudo-linear simplex defined as above. (Note that each simplex in $(D_{m(\sigma_i)})_{\fin}(\sigma_i)$ is the concatenation of $\sigma_i$ and a standard map of a simplex in $\Delta^n$.) For $\rho$, consider the sum
\begin{equation*}
  \sum_{i\in I}\left(S^{m(\sigma_i)}_{\fin}(\sigma_i) + (D_{m(\sigma_i)})_{\fin}(\partial\sigma_i) - D_{\fin}(\partial\sigma_i)\right)
\end{equation*}
and again replace each simplex in it by the composition of $\sigma_i$ and a $\sigma_i$-pseudo-linear simplex so as to obtain $\rho(c)$.

We need to show that $D(c)$ and $\rho(c)$ are indeed chains, i.e.\ that they have a standard representation. For both sums the underlying families of simplices are admissible as the family $(\sigma_i)_{i\in I}$ is and both $D(c)$ and $\rho(c)$ consist of finitely many restrictions of each $\sigma_i$ (with their $0$-faces mapped to $X$). Now $D(c)$ clearly has a standard representation since each of its simplices can be written as $\sigma_i\circ\tau$ with $\tau:\Delta^{n+1}\to \Delta^n$ and thus is degenerate. A standard representation of $\rho(c)$ can be found by combining standard representations of $\partial D(c)$, $D(\partial c)$, and $c$, according to~\eqref{chainDrho}. Hence $D(c)$ and $\rho(c)$ are chains.

\begin{proof}[Proof of Theorem~\ref{excision}.]
  Since the inclusion $\iota:C_n(A+B)\hookrightarrow C_n(X)$ maps chains in $A$ to chains in $A$ it induces a homomorphism $C_n(A+B,A)\to C_n(X,A)$. By~\eqref{chainDrho} and~\eqref{chainrhoiota} we obtain that for an $n$-cycle $z$ in $C_n(A+B,A)$ or in $C_n(X,A)$ the sum $(\rho\circ\iota)(z)-z$, respectively $(\iota\circ\rho)(z)-z$, is a boundary. Hence we have $\rho_*\circ\iota_*=\id$ and $\iota_*\circ\rho_*=\id$ and thus $\iota_*:H_n(A+B,A)\to H_n(X,A)$ is an isomorphism.
  
  We claim that the map $C_n(B)/C_n(A\cap B)\to C_n(A+B)/C_n(A)$ induced by inclusion is an isomorphism and thus induces an isomorphism $H_n(B,A\cap B)\to H_n(A+B,A)$. Then we will have $H_n(B,A\cap B)\simeq H_n(X,A)$ as desired. Indeed, $C_n(A+B)/C_n(A)$ can be obtained by starting with $C_n(B)$ and factoring out those chains whose reduced representation consists of simplices living in $\hat A$ (and hence in $\hat A\cap \hat B$). By Lemma~\ref{chaininA} and the fact that the boundary of each $\hat A_k$ is a compact subset of $X_k$, the latter are precisely the chains in $C_n(A\cap B)$, hence the map $C_n(B)/C_n(A\cap B)\to C_n(A+B)/C_n(A)$ is an isomorphism.
\end{proof}

The last axiom follows directly from the definition.

\begin{theorem}[Disjoint unions]
  For a disjoint union $X=\bigsqcup_{\alpha}X_\alpha$ (with $\hat X$ the disjoint union of all $\hat X_{\alpha}$) with inclusions $\iota_{\alpha}:X_{\alpha}\hookrightarrow X$, the direct sum map $\bigoplus_{\alpha}\left(\iota_{\alpha}\right)_*:\bigoplus_{\alpha}H_n(X_{\alpha},A_{\alpha}) \to H_n(X,A)$, where $A=\bigsqcup_{\alpha}A_{\alpha}$, is an isomorphism.%
   \COMMENT{}
\noproof
\end{theorem}

\section{Cohomology}

The cohomology belonging to the homology constructed in Section~\ref{sec:new} is defined as usual by dualization.%
   \COMMENT{}
In order to be a cohomology theory, the cochain complex has to satisfy axioms dual to those of a homology theory:
\begin{description}
\item[Homotopy equivalence:]
  If continuous maps $f,g:(X,A)\to(Y,B)$ are homotopic, then $f^*=g^*:H^n(Y,B;G)\to H^n(X,A;G)$.
\item[The Long Exact Sequence of a Pair:]
  For each pair $(X,A)$ there are coboundary homomorphisms $\delta:H^n(A)\to H^{n+1}(X,A)$ such that
  \begin{equation*}
    \begin{xy}
      \xymatrix{
	\dotsb \ar[r]^{\delta\quad\quad} & H^n(X,A;G) \ar[r]^{\;\;\pi^*} & H^n(X;G) \ar[r]^{\iota^*} & H^n(A;G) \ar[dl]_{\delta} & \\
	& & H^{n+1}(X,A;G) \ar[r]^{\;\;\pi^*} & H^{n+1}(X;G) \ar[r]^{\quad\quad\iota^*} & \dotsb
      }
    \end{xy}
  \end{equation*}
  is an exact sequence, where $\iota$ denotes the inclusion $(A,\es)\to (X,\es)$ and $\pi$ denotes the inclusion $(X,\es)\to (X,A)$. These coboundary homomorphisms are \emph{natural}, i.e.\ given a standard map $f:(X,A)\to(Y,B)$ the diagrams
  \begin{equation*}
    \begin{xy}
      \xymatrix{
        H^n(B) \ar[r]^{\delta\quad} \ar[d]^{f^*} & H^{n+1}(Y,B) \ar[d]^{f^*}\\
	H^n(A) \ar[r]^{\delta\quad} & H^{n+1}(X,A)
      }
    \end{xy}
  \end{equation*}
  commute.
\item[Excision:]
  If $(X,A)$ is an admissible pair and $B$ is a subspace of $X$ such that the interiors of $\hat A$ and $\hat B$ cover $\hat X$, the inclusion $(B,A\cap B)\hookrightarrow(X,A)$ induces isomorphisms $H^n(X,A;G)\to H^n(B,A\cap B;G)$ for all $n$.
\item[Disjoint unions:]
  For a disjoint union $X=\bigsqcup_{\alpha}X_\alpha$ with inclusions $\iota_{\alpha}:X_{\alpha}\hookrightarrow X$, the product map $\prod_{\alpha}\left(\iota_{\alpha}\right)^*:\prod_{\alpha}H^n(X_{\alpha},A_{\alpha};G) \to H^n(X,A)$, where $A=\bigsqcup_{\alpha}A_{\alpha}$, is an isomorphism.
\end{description}

The proofs of the first and the last axiom are obtained by direct dualization of the proof of the corresponding axiom for homology.%
   \COMMENT{}
The existence of a long exact sequence is also dual to the homology case: The short exact sequence
\begin{equation*}
  \begin{xy}
	\xymatrix{
	  0 \ar[r] & C_n(A) \ar[r]^{\iota_\sharp} & C_n(X) \ar[r]^{\pi_\sharp\;\;} & C_n(X,A) \ar[r] & 0
	}
  \end{xy}
\end{equation*}
dualizes to
\begin{equation*}
  \begin{xy}
	\xymatrix{
	  0 & \ar[l] C^n(A;G) & \ar[l]_{\;\;\iota^\sharp} C^n(X;G) & \ar[l]_{\pi^\sharp} C^n(X,A;G) & \ar[l] 0
	}
  \end{xy},
\end{equation*}
where $\iota^\sharp$ and $\pi^\sharp$ denote the cochain maps induced by the inclusions $\iota:(A,\es)\to(X,\es)$ and $\pi:(X,\es)\to(X,A)$. (Note that the cochain maps are the duals of the corresponding chain maps $\iota_\sharp$ and $\pi_\sharp$.) This short sequence is exact: Injectivity of $\pi^\sharp$ is immediate%
   \COMMENT{}
and so is $\ker \iota^\sharp = \im \pi^\sharp$.%
   \COMMENT{}
The surjectivity of $\iota^\sharp$ follows easily from Lemma~\ref{chaininA}.%
   \COMMENT{}
We thus have a short exact sequence of cochain complexes. Like in the homology case, this gives us the desired long exact sequence.

The last axiom, excision, follows with a proof mostly dual to that in the homology case: The chain homotopy $D$ and the chain maps $\rho$ and $\iota$ that satisfy~\eqref{chainDrho} and~\eqref{chainrhoiota} induce dual maps $D^*$, $\rho^*$, and $\iota^*$ that satisfy the dual equations $\iota^*\rho^*=\id$ and $\id - \rho^*\iota^* = D^*\delta + \delta D^*$. Therefore, $\iota^*$ and $\rho^*$ induce isomorphisms between the cohomology groups $H^n(X;G)$ and $H^n(A+B;G)$. The inclusion $\iota:C_n(A+B)\hookrightarrow C_n(X)$ is the identity on $C_n(A)$ and hence induces an inclusion $C_n(A+B,A)\hookrightarrow C_n(X,A)$ which we also denote by $\iota$. Now by the long exact sequence axiom we have a commutative diagram
\begin{equation*}
  \begin{xy}
	\xymatrix{
	  H^{n-1}(X;G) \ar[d]\ar[r]^{\iota^*} & H^{n-1}(A+B;G) \ar[d]\\
	  H^{n-1}(A;G) \ar[d]\ar[r]^{\iota^*} & H^{n-1}(A;G) \ar[d]\\
	  H^n(X,A;G) \ar[d]\ar[r]^{\iota^*} & H^n(A+B,A;G) \ar[d]\\
	  H^n(X;G) \ar[d]\ar[r]^{\iota^*} & H^n(A+B;G) \ar[d]\\
	  H^n(A;G) \ar[r]^{\iota^*} & H^n(A;G)\\
	}
  \end{xy}
\end{equation*}
and since the two upmost maps $\iota^*$ as well as the two downmost $\iota^*$ are isomorphisms, the Five Lemma~\cite{Hatcher} shows that $\iota$ induces an isomorphism $H^n(X,A;G)\to H^n(A+B,A;G)$. Since the map $C_n(B)/C_n(A\cap B) \hookrightarrow C_n(A+B)/C_n(A)$ induced by inclusion is an isomorphism, this also induces an isomorphism $C^n(A+B,A;G)\to C^n(B,A\cap B;G)$ and hence an isomorphism on cohomology. We thus have an isomorphism $H^n(X,A;G)\to H^n(B,A\cap B;G)$ as desired.

\section{The new homology for graphs}\label{sec:graphs}

In this section we wind up the analysis of our new homology theory in
the case of graphs by computing its homology groups for the case that
the space $X$ is a locally finite graph and $\hat X$ its Freudenthal
compactification. This will in particular imply
Theorem~\ref{newH1isC}.

The group homomorphism $f\colon H_1(X)\to\C(X)$ needed for
Theorem~\ref{newH1isC} counts how often the simplices in a
representative of a homology class $h$ traverse every given edge $\ve$
and then lets $f(h)$ map $\ve$ to this number. Formally, $f$ is
defined as follows. By $\eta_k$, we denote the loop that goes $k$
times around $S^1$, i.e.\ $\eta_k(t)= \ee^{2\pi\ii kt}$. For every
edge $e$ of $X$, let $f_e\colon \hat X\to S^1$ wrap $e$ round $S^1$ in
its natural direction, defining $f_e\restr e$ as
$\eta_1\circ\theta_e^{-1}$ (recall that $\theta_e\colon[0,1]\to e$ is
the homeomorphism given by the definition of $X$ as a $1$-complex) and
putting $f_e(\hat X\setminus e)\assign 1$. Now let $[c]\in H_1(X)$.
Only finitely many simplices in the reduced representation of $c$ meet
$e$, let $c'$ be the chain represented by the sum of these simplices.
The chain $\fes(c')$ is a $1$-cycle in $S^1$, hence its homology class
is represented by some $\eta_k =: \eta_{k(c,e)}$. We now define
$f([c])$ by putting $f([c])(\ve)\assign k(c,e)$.

The first thing to check is whether $f$ is well defined. To this
end, let $c_1,c_2$ be representatives of the same homology
class. Then $c_1-c_2$ is the sum of boundaries of an admissible family
of $2$-simplices. Since only finitely many of these $2$-simplices, say
$\sigma_1,\dotsc,\sigma_n$, can meet $e$, the $1$-cycles $\fes(c'_1)$
and $\fes(c'_2)$ in $S^1$ differ by a sum of finitely many
boundaries%
  \COMMENT{}
and are thus homologous, implying $k(c_1,e)=k(c_2,e)$. Therefore,
$f(h)$ is well defined.

The map $f$ defined here is a homomorphism $H_1(X)\to\vEE(X)$, but it
is in fact even a homomorphism $H_1(X)\to\C(X)$, which follows
immediately from Proposition~\ref{representH1} and the
fact that the corresponding map for standard singular homology has its
image in $\C(X)$ \cite[Lemma~11]{Hom1}.

We are now ready to state the main theorem of this section.

\begin{theorem}\label{thm:Hngraph}
  Let $X$ be a locally finite connected graph and $\hat X$ its
  Freudenthal compactification. Then
  \begin{enumerate}
  \item\label{enum:H0graph}
    mapping a $0$-chain in $X$ with finite reduced representation
    $\sum\lambda_i\sigma_i$ to $\sum\lambda_i$ defines a homomorphism
    $C_0(X)\to\ZZ$ which in turn induces an isomorphism
    $H_0(X)\to\ZZ$,
  \item\label{enum:H1graph}
    the homomorphism $f\colon H_1(X)\to\C(X)$ defined above is an
    isomorphism, and
  \item\label{enum:Hngraph}
    $H_n(X)=0$ for every $n>1$.
  \end{enumerate}
\end{theorem}

\begin{proof}
  \begin{enumerate}
  \item[\ref{enum:H0graph}]
    As no $0$-simplex is degenerate and every degenerate $1$-simplex
    is constant and hence a cycle, we obtain that every $0$-chain has
    a finite reduced representation, hence the map defined above is
    indeed a homomorphism $C_0(X)\to\ZZ$. Moreover, the boundaries of
    $1$-chains are precisely the boundaries of the finite $1$-chains
    and hence the group $H_0(X)$ is the same as in standard singular
    homology. In particular, the above map induces an isomorphism to
    $\ZZ$.
  \item[\ref{enum:H1graph}]
    The map $f$ is surjective since the corresponding map for singular
    homology is~\cite[Lemma 12]{Hom1}. To show that it is injective,
    let $c$ be a finite $1$-cycle with $f([c])=0$, that is, every edge
    of $X$ is traversed by the simplices in $c$ the same number of
    times in both directions. In a finite graph, we would subdivide
    the simplices into their passes through the edges of $X$, thus
    showing that $c$ is null-homologous. In an infinite graph, we
    would have to subdivide infinitely often, which is not possible in
    standard singular homology. But in our new homology, we can: There
    is an admissible sum of $2$-simplices whose boundary we can add to
    $c$ so as to obtain a sum $c'$ of passes through
    edges~\cite[Lemma~20]{Hom1}. Since $\hat X$ is $1$-dimensional,
    all $2$-simplices in $\hat X$ are degenerate, implying that the
    above sum of $2$-simplices is good and hence $c'$ is a
    $1$-cycle. As $c$ and $c'$ are homologous, we have $f([c']) =
    f([c]) = 0$. Thus for every edge $e$, the cycle $c'$ contains the
    same number of passes through $\ve$ as through $\ev$, showing that
    $c'$, and hence also $c$, is null-homologous.
  \item[\ref{enum:Hngraph}]
    Let $n>1$ and an $n$-cycle $z$ be given; we show that $z$ is a
    boundary. To this end, choose an enumeration $e_0,e_1,\dotsc$ of
    the edges of $X$. Let $B_1$ be the union of $X-e_0$ and two
    disjoint closed half-edges of $e_0$, one at each endvertex. Then
    the interiors of $e_0$ and $\hat B_1$ cover $\hat X$. We may thus
    apply excision.
    
    Let $\rho$ be the map $C_n(X)\to C_n(e_0+B_1)$ from
    \eqref{chainDrho} and~\eqref{chainrhoiota}. Then $\rho(z)$
    is the sum of a chain in $e_0$ and a chain in $B_1$. The boundary
    of both of those chains is an ($n-1$)-cycle in $e_0\cap B_1$. As
    $e_0\cap B_1$ is the disjoint union of two closed intervals, all
    its homology groups of dimension at least $1$ vanish, hence the
    boundary of the two chains is also a boundary in $e_0\cap B_1$.
    Choose an $n$-chain in $e_0\cap B_1$ with the right boundary and
    subtract it from our two chains in $e_0$ and $B_1$ so as to obtain
    cycles $z_0$ in $e_0$ and $z'_1$ in $B_1$. Note that $z_0+z'_1$ is
    homologous to $z\sasign z'_0$.
    
    Now repeat the construction with $z'_1$, $e_1$, and $B_2$ the
    union of $B_1-e_1$ and two disjoint half-edges of $e_1$ so as to
    obtain cycles $z_1$ in $e_1$ and $z'_2$ in $B_2$. Working through
    the edges $e_i$ in turn, we obtain cycles $z_i$ in $e_i$
    and $z'_{i+1}$ in $B_{i+1}$. Since $X$ is locally finite, for
    every vertex $v$ there exists an $i$ such that the component $C_v$
    of $B_i$ containing $v$ is a closed star around $v$. In all later
    $B_j$, this component remains unchanged, and hence the simplices
    of $z'_i$ living in $C_v$ are not touched by $\rho$, i.e.\ all
    later $z'_j$ agree on $C_v$; let $z_v$ be the cycle in $C_v$
    formed by those simplices.
    
    Since each $z'_i$ is homologous in $B_i$ to $z_i+z'_{i+1}$ (with
    $B_0\assign X$), the family of all simplices in the ($n+1$)-chains
    certifying these homologies is locally finite in $X$: For every
    $x\in X$ there is an $i$ such that either $x\notin B_i$ or $x$ is
    contained in a component $C_v$ of $B_i$ (if $x$ is a vertex, then
    obviously $v=x$). In either case there is a neighbourhood around
    $x$ that avoids all the ($n+1$)-chains of later steps. Since each
    ($n+1$)-simplex is degenerate, the family of those simlices is
    admissible.
    
    Thus, $z$ is homologous to the sum of all $z_i$ and $z_v$. Since
    each $e_i$ and each $C_v$ has trivial homology in dimension $n$,
    each $z_i$ is a boundary in $e_i$, of an ($n+1$)-chain $c_i$ say,
    and so is each $z_v$ in $C_v$, of an ($n+1$)-chain $c_v$ say. As
    every point in $X$ has a neighbourhood that meets only finitely
    many $e_i$ and $C_v$, the infinite sum $\sum_i c_i + \sum_v c_v$
    is an ($n+1$)-chain $c$ in $X$. By construction, $\partial c = z$.
  \end{enumerate}
\end{proof}

Note that Theorem~\ref{thm:Hngraph}\ref{enum:H0graph} holds for
every connected locally compact Hausdorff space $X$. Hence $H_0(X) =
\bigoplus_{C\in\mathfrak C}\Z$, where $\mathfrak C$ is the set of
components of $\hat X$.

\bibliographystyle{amsplain}
\bibliography{collective}

\providecommand{\bysame}{\leavevmode\hbox to3em{\hrulefill}\thinspace}
\providecommand{\MR}{\relax\ifhmode\unskip\space\fi MR }
\providecommand{\MRhref}[2]{%
  \href{http://www.ams.org/mathscinet-getitem?mr=#1}{#2}
}
\providecommand{\href}[2]{#2}
\begin{thebibliography}{10}

\bibitem{AbelsStrantzalos}
H.~Abels and P.~Strantzalos, \emph{Proper transformation groups}, in
  preparation.

\bibitem{BergerBruhnDeg}
E.~Berger and H.~Bruhn, \emph{Eulerian edge sets in locally finite graphs},
  preprint 2008.

\bibitem{LocFinTutte}
H.~Bruhn, \emph{The cycle space of a $3$-connected locally finite graph is
  generated by its finite and infinite peripheral circuits}, J.~Combin.\ Theory
  (Series B) \textbf{92} (2004), 235--256.

\bibitem{Duality}
H.~Bruhn and R.~Diestel, \emph{Duality in infinite graphs}, Comb.,\ Probab.\
  Comput. \textbf{15} (2006), 75--90.

\bibitem{Partition}
H.~Bruhn, R.~Diestel, and M.~Stein, \emph{Cycle-cocycle partitions and faithful
  cycle covers for locally finite graphs}, J.~Graph Theory \textbf{50} (2005),
  150--161.

\bibitem{bicycle}
H.~Bruhn, S.~Kosuch, and M.~Win Myint, \emph{Bicycles and left-right tours in
  locally finite graphs}, Preprint 2007.

\bibitem{LocFinMacLane}
H.~Bruhn and M.~Stein, \emph{Mac{L}ane's planarity criterion for locally finite
  graphs}, J.~Combin.\ Theory (Series B) \textbf{96} (2006), 225--239.

\bibitem{CyclesIntro}
R.~Diestel, \emph{The cycle space of an infinite graph}, Comb.,\ Probab.\
  Comput. \textbf{14} (2005), 59--79.

\bibitem{DiestelBook05}
\bysame, \emph{Graph {T}heory \emph{(3rd edition)}}, Springer-Verlag, 2005, \\
  Electronic edition available at:\\ {\small\tt
  http://www.math.uni-hamburg.de/home/diestel/books/graph.theory}.

\bibitem{CyclesI}
R.~Diestel and D.~K{\"u}hn, \emph{On infinite cycles {I}}, Combinatorica
  \textbf{24} (2004), 68--89.

\bibitem{Hom1}
R.~Diestel and P.~Spr\"ussel, \emph{The homology of locally finite graphs with
  ends}, Combinatorica (to appear).

\bibitem{Engelking}
R.~Engelking, \emph{Theory of {D}imensions, {F}inite and {I}nfinite}, Sigma
  series in pure mathematics, vol.~10, Heldermann Verlag, 1995.

\bibitem{hotchpotch}
A.~Georgakopoulos, \emph{Topological circles and {E}uler tours in locally
  finite graphs}, Electronic J.\ Comb. \textbf{16} (2009), {\#\/R40}.

\bibitem{geo}
A.~Georgakopoulos and P.~Spr\"ussel, \emph{Geodesic topological cycles in
  locally finite graphs}, preprint 2006.

\bibitem{Hatcher}
A.~Hatcher, \emph{Algebraic {T}opology}, Cambrigde Univ.\ Press, 2002.

\bibitem{EilSteen}
N.~Steenrod S.~Eilenberg, \emph{Foundations of algebraic topology}, Princeton
  Univ. Press, 1952.

\bibitem{Arboricity}
M.~Stein, \emph{Arboriticity and tree-packing in locally finite graphs},
  J.~Combin.\ Theory (Series B) \textbf{96} (2006), 302--312.

\end{thebibliography}

\small
\parindent=0pt
\vskip2mm plus 1fill

\begin{tabular}{cc}
\begin{minipage}[t]{0.5\linewidth}
Reinhard Diestel\\ {\tt}\\
Mathematisches Seminar\\
Universit\"at Hamburg\\
Bundesstra\ss e 55\\
20146 Hamburg\\
Germany\\
\end{minipage} 
&
\begin{minipage}[t]{0.5\linewidth}
Philipp Spr\"ussel\\ {\tt <Philipp.Spruessel@gmx.de>}\\
Mathematisches Seminar\\
Universit\"at Hamburg\\
Bundesstra\ss e 55\\
20146 Hamburg\\
Germany\\
\end{minipage}
\end{tabular} 

\smallskip
Version 22.5.2011
\end{document}